\documentclass[a4paper,12pt,reqno]{amsart}

\usepackage{amsmath}
\usepackage{amscd}
\usepackage{amssymb}
\usepackage{mathrsfs}
\usepackage[left=2.5cm,right=2.5cm,bottom=3cm,top=3cm]{geometry}
\newtheorem{thm}{Theorem}[section]
\newtheorem{cor}[thm]{Corollary}
\newtheorem{exmp}[thm]{Example}
\newtheorem{lem}[thm]{Lemma}

\theoremstyle{definition}

\theoremstyle{remark}
\newtheorem{rem}[thm]{Remark}

\numberwithin{equation}{section}
\begin{document}
\title{The continuity method on Fano fibrations}
\author{Yashan Zhang}
\address{Beijing International Center for Mathematical Research, Peking University, Beijing, China}
\email{yashanzh@pku.edu.cn}
\thanks{Y. Zhang is partially supported by the Science and Technology Development Fund (Macao S.A.R.) Grant FDCT/ 016/2013/A1 and the Project MYRG2015-00235-FST of the University of Macau}

\author{Zhenlei Zhang}
\address{Department of mathematics, Capital Normal University, Beijing, China}
\email{zhleigo@aliyun.com}
\thanks{Z. Zhang is partially supported by NSFC 11431009}

\begin{abstract}
We study finite-time collapsing limits of the continuity method. When the continuity method starting from a rational initial K\"ahler metric on a projective manifold encounters a finite-time volume collapsing, this projective manifold admits a Fano fibration over a lower dimensional base. In this case, we prove the continuity method converges to a singular K\"ahler metric on the base in the weak sense; moreover, if the base is smooth and the fibration has no singular fibers, then the convergence takes place in Gromov-Hausdorff topology.
\end{abstract}

\maketitle
\section{Motivation and main result}
In \cite{LT}, La Nave and Tian introduced a new approach to the Analytic Minimal Model Program. It is a continuity method of complex Monge-Amp\`ere equations.

\par Let $X$ be an $n$-dimensional projective manifold with an ample $\mathbb Q$-line bundle $L$. For any fixed K\"ahler metric $\omega_0\in2\pi c_1(L)$, we consider the following continuity method on $X$ introduced by La Nave and Tian in \cite{LT} (also see Rubinstein \cite{Ru}):
\begin{equation}\label{lt.0}
\left\{
\begin{aligned}
\omega(t)&=\omega_0-tRic(\omega(t))\\
\omega(0)&=\omega_0.
\end{aligned}
\right.
\end{equation}
It is proved in \cite[Theorem 1.1]{LT} that the maximal existence time of (\ref{lt.0}) is

\begin{center}
$T=\sup\{t>0|[\omega_0]-2\pi tc_1(X)>0\}=\sup\{t>0|L+tK_X$ is ample$\}$.\\
\end{center}

According to the Analytic Minimal Model Program proposed in \cite{LT}, there are several independent cases to consider.
\begin{itemize}
\item[(1)] $T=\infty$. Then $X$ is a minimal model and it is conjectured that, after a suitable normalization, the continuity method should converge in Gromov-Hausdorff topology to a generalized K\"ahler-Einstein metric on the canonical model of $X$, see \cite{LT} for detailed descriptions and \cite{FGS,LTZ,ZyZz} for some progresses.
\item[(2)] $T<\infty$. By \cite[Theorem 1.1, Corollary 2.2]{LT} we know $T<\infty$ means $K_X$ is not nef (or $c_1(X)$ contains some ``positive part"). In this case, in general it is expected that the continuity method will contract/collapse certain ``positive part" of $c_1(X)$ in Gromov-Hausdorff topology. To be precise, we recall some facts from algebraic geometry. By the rationality theorem (see e.g. \cite{Ma}) we know $T\in\mathbb Q$. We always assume without loss of any generality that $T=1$. By the base-point-free theorem (see e.g. \cite{Ma}) we know the limiting $\mathbb Q$-line bundle $L+K_X$ is semi-ample. In this case, we have a holomorphic map
\begin{equation}\label{fibration}
f:X\to Y\subset\mathbb{CP}^N
\end{equation}
induced by the linear system of $m(L+K_X)$ for some sufficiently large integer $m$ (see e.g. \cite{La}). Here $Y=f(X)$ is an irreducible normal projective variety and $k:=dim(Y)$ equals the Iitaka dimension of $L+K_X$. We know $k\in\{0,1,\ldots,n\}$. Let's separate discussions into several cases as follows.
\begin{itemize}
\item[(2.1)] $k=0$ (finite-time extinction). Then $Y$ is a single point and hence $X$ is in fact a Fano manifold with $L=-K_X$. In this case, the continuity method \eqref{lt.0} starting from any K\"ahler metric $\omega_0\in2\pi c_1(L)=2\pi c_1(X)$ collapses to a point in Gromov-Hausdorff topology. In fact, in this case it is natural to consider a normalized version of \eqref{lt.0} as follows
\begin{equation}\label{lt.5}
(1-t)\omega(t)=\omega_0-tRic(\omega(t)).
\end{equation}
By \cite[Theorem 1.1]{LT}, \eqref{lt.5} has a unique smooth solution $\omega(t)$ for $t\in[0,1)$. Moreover, by using some well-known results in complex Monge-Amp\`{e}re equations (see e.g. \cite[Section 3.1]{Ch} for detailed discussions) we know that, as $t\to1^-$, $\omega(t)$ solving \eqref{lt.5} converges in $C^\infty(X,\omega_0)$-topology to a unique K\"ahler metric $\omega(1)\in2\pi c_1(X)$ satisfying $Ric(\omega(1))=\omega_0$. Then the time $t$ will go through $1$ and \eqref{lt.5} can be written as
 \begin{equation}\label{lt.5.1}
 Ric(\omega(t))=\frac{1}{t}\omega_0+(1-\frac{1}{t})\omega(t),
 \end{equation}
 which is a simple reparametrization of the classical continuity method proposed by Aubin \cite{Au,Au84}. Therefore, all the results concerning Aubin's continuity method on Fano manifolds apply in our case. For example, by the results in \cite{BM}, if $X$ admits a unique K\"ahler-Einstein metric $\omega_{KE}\in2\pi c_1(X)$, then \eqref{lt.5} (or \eqref{lt.5.1}) is solvable for $t\in[1,\infty)$ and $\omega(t)$ converges in $C^\infty(X,\omega_0)$-topology to $\omega_{KE}$ as $t\to\infty$.\\

\item[(2.2)] $k=n$ (finite-time volume non-collapsing). In this case, according to results in \cite{LT,LTZ}, there exist a proper subvariety $V\subset Y$ and a K\"ahler metric $\omega_Y$ on $Y\setminus V$ such that $\omega(t)\to f^*\omega_Y$ in $C^\infty_{loc}(f^{-1}(Y\setminus V))$-topology and $(X,\omega(t))$ converges in Gromov-Hausdorff topology to the metric completion of $(Y\setminus V,\omega_Y)$, which is a compact length metric space homeomorphic to $Y$.\\

\item[(2.3)] $1\le k\le n-1$ (finite-time volume collapsing). In this case, $f:X\to Y$ is a Fano fibration and it is conjectured (see \cite[Conjecture 4.1]{LT}) that \emph{$(X,\omega(t))$ should converge in Gromov-Hausdorff topology to a compact metric $d_Y$ on $Y$; moreover $d_Y$ is a ``nice" K\"ahler metric $\omega_Y$ on $Y\setminus S$ for some proper subvariety $S$ of $Y$ and $\omega(t)$ should converge to $f^*\omega_Y$ on $f^{-1}(Y\setminus S)$ in some more regular topology, possibly, the smooth topology. }
\end{itemize}

\end{itemize}
\par In this paper, we will focus on the above finite-time volume collapsing case (2.3) and partially confirm conjecture in this case. More precisely, we will prove the following main results.

\begin{thm}\label{t1}
Assume $X$ is an $n$-dimensional projective manifold with an ample $\mathbb Q$-line bundle $L$ and $T:=\sup\{t>0|L+tK_X$ is ample$\}=1$. Let $f:X\to Y\subset \mathbb{CP}^N$ be the map given in \eqref{fibration} with $1\le dim(Y)\le n-1$. Define a proper subvariety $S\subset Y$ be the singular set of $Y$ together with the critical values of $f$. For any K\"ahler metric $\omega_0\in2\pi c_1(L)$, let $\omega(t)_{t\in[0,1)}$ be the unique smooth solution of \eqref{lt.0} starting from $\omega_0$ on $X$. Then there exists a positive $(1,1)$-current $\omega_Y$ on $Y$, which is a K\"ahler metric on $Y\setminus S$, such that as $t\to1$, $\omega(t)\to f^*\omega_Y$ as currents on $X$. Moreover, for any $K\subset\subset Y\setminus S$, $\omega(t)\to f^*\omega_Y$ in $C^0(K,\omega_0)$-topology.
\end{thm}
The current $\omega_Y$ is canonically constructed in terms of  the fibration structure and the initial metric $\omega_0$, see Section \ref{limiting} for more discussions.

\par A direct consequence of Theorem \ref{t1} is the following

\begin{cor}\label{t2}
Assume the same as in Theorem \ref{t1} and, additionally, $S=\emptyset$, i.e. $Y$ is smooth and $f:X\to Y$ is a holomorphic submersion. Then for any K\"ahler metric $\omega_0\in2\pi c_1(L)$, there exists a K\"ahler metric $\omega_Y$ on $Y$ such that the unique solution $\omega(t)_{t\in[0,1)}$ of \eqref{lt.0} starting from $\omega_0$ on $X$ converges to $f^*\omega_Y$ in $C^0(X,\omega_0)$-topology as $t\to1$. In particular, $(X,\omega(t))\to(Y,\omega_Y)$ in Gromov-Hausdorff topology as $t\to1$.
\end{cor}

In fact, Corollary \ref{t2} can be slightly generalized as follows.
\begin{cor}\label{r1}
Let $f:X\to Y$ be a holomorphic submersion between two compact K\"ahler manifolds with $1\le dim(Y)<dim(X)$ and assume there exist a K\"ahler metric $\chi$ on $Y$ such that $f^*[\chi]+2\pi c_1(X)$ is a K\"ahler class on $X$. Then for any K\"ahler metric $\omega_0\in f^*[\chi]+2\pi c_1(X)$, there exists a K\"ahler metric $\omega_Y\in[\chi]$ on $Y$ such that $\omega(t)_{t\in[0,1)}$, the unique solution of \eqref{lt.0} starting from $\omega_0$ on $X$, converges to $f^*\omega_Y$ in $C^0(X,\omega_0)$-topology as $t\to1$. In particular, $(X,\omega(t))\to(Y,\omega_Y)$ in Gromov-Hausdorff topology as $t\to1$.
\end{cor}
If we check step by step, then it is clear that Corollary \ref{r1} can be proved by the same arguments for Corollary \ref{t2} (or Theorem \ref{t1}). We remark that, in the setting of Theorem \ref{t1}, the rationality of $[\omega_0]$ is used to provide the Fano fibration \eqref{fibration} and the rationality of $[\chi]$ (see \eqref{a1} in Section 2) is used to obtain higher order regularity of $\omega_Y$ on $Y\setminus S$ (where one needs Kodaira lemma, which holds for nef and big $\mathbb Q$-line bundle, see \cite{ST12} for details). Therefore, if we are given the setting in Corollary \ref{r1}, then we don't need to assume $[\chi]$ nor $[\omega_0]$ to be rational.\\

\begin{rem}
Very recently, Fu, Guo and Song \cite{FGS} made a big progress on studying the geometry of the continuity method. They proved that the diameter of $\omega(t)$ is uniformly bounded. Combining with our Theorem \ref{t1}, it seems very likely that the metrics $(X,\omega(t))$ converge to the limiting metric constructed in section 2 below. We shall return to this problem later.
\end{rem}

Let's look at an example.

\begin{exmp}\label{exmp2}
Let $\Sigma_a$, $a\in\mathbb{Z}_{\ge1}$, be a complete smooth fan in $\mathbb{R}^2$ with the minimal generators $u_0=(-1,a)$, $u_1=(0,1)$, $u_2=(1,0)$ and $u_3=(0,-1)$. Let $\mathcal{H}_a$ be the smooth toric variety corresponding to $\Sigma_a$, which is called the $a$-th Hirzebruch surface. We shall explain how to equip a $\mathbb{CP}^1$-bundle structure on $\mathcal{H}_a$ by extremal contractions. To this end, firstly note that the wall $\tau=\mathbb{R}_{\ge0}\cdot u_2$ is an extremal wall and will give an extremal ray $\mathcal{R}$ in the Mori cone of $\mathcal{H}_a$ (see \cite[Example 6.3.23]{CLS}). Moreover, since the wall relation of $\tau$ is
\begin{equation}
u_1+0\cdot u_2+u_3=0\nonumber,
\end{equation}
by \cite[Proposition 15.4.5]{CLS} the extremal contraction given by $\mathcal{R}$ is a fibration (i.e., Mori fiber space) $f:\mathcal{H}_a\to\mathbb{CP}^1$ and the fibers are isomorphic to the toric variety of the complete fan in $\mathbb{R}$, that is $\mathbb{CP}^1$. Hence all fibers are smooth and $f$ is in fact a locally trivial fibration (see \cite[p.190]{BHPV}). Consequently, $f$ is a $\mathbb{CP}^1$-bundle. Now the discussions in \cite[Section V.4, esp. Proposition V.4.2]{BHPV} imply that $f$ is exactly the projective bundle $\mathbb{P}(\mathcal{O}_{\mathbb{CP}^1}\oplus\mathcal{O}_{\mathbb{CP}^1}(a))\to\mathbb{CP}^1$, which is the desired conclusion.
\par Next,  by \cite[Lemma 2.2]{ToZy16} we can find a suitable constant $t_0$ such that $2\pi t_0c_1(\mathcal{H}_a)+f^*(2\pi c_1(\mathbb{CP}^1))$ is a K\"ahler class on $\mathcal{H}_a$. Then by Corollary \ref{t2} or Remark \ref{r1} the continuity method \eqref{lt.0} starting from any K\"ahler metric $\omega_0\in2\pi t_0c_1(\mathcal{H}_a)+f^*(2\pi c_1(\mathbb{CP}^1))$ converges to a K\"ahler metric $\bar\omega\in2\pi c_1(\mathbb{CP}^1)$ on $\mathbb{CP}^1$ in Gromov-Hausdorff topology. By case (2.1) the continuity method restarting from $\bar\omega$ on $\mathbb{CP}^1$ will converge to a point in Gromov-Hausdorff topology. Moreover, if we deform $\bar\omega$ by the normalized version \eqref{lt.5}, then it will converge in smooth topology on $\mathbb{CP}^1$ to, up to a biholomorphism, $2\omega_{FS}$ as $t\to\infty$, where $\omega_{FS}$ is the Fubini-Study metric on $\mathbb{CP}^1$ with $Ric(\omega_{FS})=2\omega_{FS}$.
\end{exmp}

\begin{rem}
The finite time collapsing of the K\"ahler-Ricci flow has been studied in many papers, see \cite{F1,F2,FuZs,SeT,So,SSW,SW,SY,TZ,ToZy16} and the references therein. For example, the collapsing of K\"ahler-Ricci flow on Hirzebruch surfaces is studied in \cite{SW} (also see \cite{F1} for certain generalizations), where the similar picture as in Example \ref{exmp2} is obtained by assuming certain symmetry condition on initial metrics. From the view point of Analytic Minimal Model Program with Ricci flow (in particular, see \cite[Conjecture 6.6]{ST16}), the result similar to Theorem \ref{t1} should be true for K\"ahler-Ricci flow.
\end{rem}

The rest of this paper is organized as follows. We will construct the limiting metric $\omega_Y$ on $Y$ in Section \ref{limiting}. Then we prove a weak convergence in Section \ref{est section} and uniform convergence of metric away from singular fibers in Section \ref{pf}.

\section{Construction of limiting metrics}\label{limiting}

Assume the same as in Theorem \ref{t1}. Then the limiting class satisfies
\begin{equation}\label{a1}
[\omega_0]-2\pi c_1(X)=f^{*}[\chi],
\end{equation}
where $\chi=\frac{1}{m}\omega_{FS}$ for some positive integer $m$. Here $\omega_{FS}\in2\pi c_1(\mathcal O_{\mathbb{CP}^N}(1))$ is the Funibi-Study metric on $\mathbb{CP}^N$.
\par \eqref{a1} in particular implies, for $y\in Y\setminus S$,
\begin{equation}\label{a2}
\omega_{0}|_{X_y}\in2\pi c_1(X_y).
\end{equation}

By \eqref{a1} we fix a smooth positive volume form $\Omega$ on $X$ with
\begin{equation}\label{vol}
\sqrt{-1}\partial\bar{\partial}\log\Omega=(f^*\chi-\omega_0).
\end{equation}

\par For $y\in Y\setminus S$, we denote $\omega_{0,y}:=\omega_0|_{X_y}$. Then by using \eqref{a2} and the $\partial\bar{\partial}$-lemma one can choose a $\rho_{y}\in C^\infty(X_y,\mathbb{R})$ with
\begin{equation}\label{}
\left\{
\begin{aligned}
Ric(\omega_{0,y})-\omega_{0,y}&=\sqrt{-1}\partial_f\bar{\partial}_f\rho_y\\
\int_{X_y}e^{\rho_{y}}(\omega_{0,y})^{n-k}&=\int_{X_y}(\omega_{0,y})^{n-k}.
\end{aligned}
\right.
\end{equation}
Here $\partial_f$ is the restriction of $\partial$ to the smooth fiber.
Define a function $\rho$ on $X_{reg}:=f^{-1}(Y\setminus S)$ by setting $\rho(y,\cdot):=\rho_y(\cdot)$. Then $\rho\in C^\infty(X_{reg},\mathbb{R})$. Moreover, by Yau \cite{Y}, we have a unique $u_y\in C^\infty(X_y,\mathbb R)$ satisfying
\begin{equation}\label{}
\left\{
\begin{aligned}
(\omega_{0,y}+\sqrt{-1}\partial_f\bar\partial_f u_y)^{n-k}&=e^{\rho_y}(\omega_{0,y})^{n-k}\\
\int_{X_y}u_{y}(\omega_{0,y})^{n-k}&=0.
\end{aligned}
\right.
\end{equation}

Define a function $u$ on $X_{reg}:=f^{-1}(Y\setminus S)$ by setting $u(y,\cdot):=u_y(\cdot)$. Then $u\in C^\infty(X_{reg},\mathbb{R})$ and $\overline\omega_0:=\omega_0+\sqrt{-1}\partial\bar\partial u$ is a closed real $(1,1)$-form on $X_{reg}$. Denote $\overline\omega_{0,y}:=\overline\omega_{0}|_{X_y}$, then $\bar\omega_{0,y}$ is a K\"ahler metric on smooth fiber $X_y$ with
\begin{equation}\label{calabi.eq}
Ric(\overline\omega_{0,y})=\omega_{0,y}.
\end{equation}

Define a function
\begin{equation}
G=\frac{\Omega}{\binom{n}{k}\overline\omega_0^{n-k}\wedge(f^*\chi)^k}\nonumber.
\end{equation}
Then $G$ can be seen as a smooth positive function on $Y\setminus S$. In fact, by direct computation:
\begin{align}
\sqrt{-1}\partial_f\bar{\partial}_f\log G&=\sqrt{-1}\partial_f\bar{\partial}_f\log\Omega+Ric(\overline\omega_{0,y})\nonumber\\
&=-\omega_{0,y}+Ric(\overline\omega_0|_{X_y})\nonumber\\
&=0\nonumber
\end{align}
by the definition of $\rho$. Hence $G$ is a constant along each smooth fiber and descends to a smooth positive function on $Y\setminus S$. In fact, as in \cite[Lemma 3.3]{ST12}, on $Y\setminus S$ we have
\begin{equation}\label{equation}
G=\frac{f_*\Omega}{\binom{n}{k}V_0\chi^k},
\end{equation}
where $V_0:=\int_{X_y}(\omega_{0,y})^{n-k}$ is a positive constant.
\par Moreover, by \cite[Proposition 3.2]{ST12} and its argument, we can find two positive constants $\delta$ and $\epsilon$ such that
\begin{equation}\label{Lp}
0<\delta\le G\in L^{1+\epsilon}(Y,\chi^k).
\end{equation}
\par Consider the following complex Monge-Amp\`{e}re equation on $Y$
\begin{equation}\label{limit}
(\chi+\sqrt{-1}\partial\bar{\partial}\psi)^k=e^{\psi}G\chi^k,
\end{equation}
where we have used the same notation $\chi$ to denote the restriction of $\chi$ on $Y$.
\par Having \eqref{Lp}, by \cite[Theorem 4.1]{EGZ} and \cite[Theorem 3.2]{ST12} (building on \cite{Y,K98}) there exists a unique solution $\psi\in PSH(Y,\chi)\cap L^\infty(Y)\cap C^\infty(Y\setminus S)$ to \eqref{limit}. We denote $\omega_Y:=\chi+\sqrt{-1}\partial\bar\partial\psi$, which is a positive $(1,1)$-current on $Y$ and a smooth K\"ahler metric on $Y\setminus S$. We will see that $\omega_Y$ is the limit of the continuity method \eqref{lt.0}.

\section{Estimates and weak convergence}\label{est section}
From now on, we will use the following reparametrization of \eqref{lt.0},
\begin{equation}\label{lt}
\left\{
\begin{aligned}
\omega(t)&=\omega_0-(1-e^{-t})Ric(\omega(t))\\
\omega(0)&=\omega_0.
\end{aligned}
\right.
\end{equation}
which will be more convenient for later discussions.
Obviously, \eqref{lt} has a unique solution $\omega(t)$ for $t\in[0,\infty)$ and the finite time collapsing of \eqref{lt.0} at $t=1$ is exactly the infinite time collapsing of \eqref{lt}. A useful fact is that for $t\in[1,\infty)$ we have the following uniform lower bound for Ricci curvature along the continuity method \eqref{lt}:
\begin{equation}\label{ric}
Ric(\omega(t))\ge-2\omega(t).
\end{equation}

\par First of all we reduce the continuity method (\ref{lt}) to a complex Monge-Amp\`{e}re equation as follows. Define
\begin{equation}
\omega_t=e^{-t}\omega_0+(1-e^{-t})f^*\chi\nonumber.
\end{equation}
Then $\omega(t)=\omega_t+\sqrt{-1}\partial\bar{\partial}\varphi(t)$ solves (\ref{lt}) if $\varphi(t)$ solves

\begin{equation}\label{lt1}
\left\{
\begin{aligned}
(\omega_t+\sqrt{-1}\partial\bar{\partial}\varphi(t))^{n}&=e^{-(n-k)t}e^{\frac{\varphi(t)}{1-e^{-t}}}\Omega\\
\varphi(0)&=0.
\end{aligned}
\right.
\end{equation}

\begin{lem}\label{C0lem}\cite{DP,EGZ08}
There exists a constant $C\ge1$ such that
$$\|\varphi(t)\|_{C^0(X\times[1,\infty))}\le C.$$
\end{lem}
\begin{proof}
Firstly note that $\omega_t^n\le Ce^{-(n-k)t}\Omega$ on $X\times[1,\infty)$ for some constant $C\ge1$. Then by the maximum principle, we easily see that $sup_{X\times[1,\infty)}\varphi(t)$ is uniformly bounded from above. Moreover, for any $t\in[1,\infty)$,
$$e^{2\sup_X\varphi(t)}\int_{X}\Omega\ge\int_{X}e^{\frac{\varphi(t)}{1-e^{-t}}}\Omega=e^{(n-k)t}\omega_t^n\ge C^{-1}$$
for some uniform constant $C\ge1$. Hence
\begin{equation}\label{C01}
\left|\sup_{X\times[1,\infty)}\varphi(t)\right|\le C.
\end{equation}
On the other hand, for an arbitrary fixed positive constant $\epsilon$ and all $t\in[1,\infty)$, we have
$$\int_X\left(\frac{e^{\frac{\varphi(t)}{1-e^{-t}}}}{e^{(n-k)t}[\omega_t]^n}\right)^{1+\epsilon}\le C,$$
where $[\omega_t]^n:=\int_X\omega_t^n$ and we have used $e^{(n-k)t}[\omega_t]^n$ is uniformly bounded from below. By applying \cite[Theorem 2.2]{DP} (also see \cite{EGZ08}), we find a constant $C\ge1$ such that for all $t\in[1,\infty)$,
\begin{equation}\label{C02}
\sup_X\varphi(t)-\inf_X\varphi(t)\le C.
\end{equation}
Combining \eqref{C01} and \eqref{C02}, Lemma \ref{C0lem} is proved.
\end{proof}
An immediate consequence of Lemma \ref{C0lem} is
\begin{lem}\label{C0lem'}
There exists a constant $C\ge1$ such that
$$C^{-1}e^{-(n-k)t}\Omega\le\omega(t)^n\le Ce^{-(n-k)t}\Omega.$$
\end{lem}

\begin{lem}\label{C2lem}
There exists a constant $C\ge1$ such that on $X\times[1,\infty)$,
$$tr_{\omega(t)}f^*\chi\le C$$
\end{lem}
\begin{proof}
By a Schwarz lemma argument (see e.g. \cite{Y78}), we find a constant $C\ge1$ such that
\begin{equation}\label{C2.1}
\Delta_{\omega(t)}\log tr_{\omega(t)}f^*\chi\ge -C tr_{\omega(t)}f^*\chi-C.
\end{equation}
On the other hand, for $t\in[1,\infty)$,
\begin{align}\label{C2.2}
\Delta_{\omega(t)}\varphi(t)&=tr_{\omega(t)}(\omega(t)-e^{-t}\omega_0-(1-e^{-t})f^*\chi)\nonumber\\
&\le n-2^{-1}tr_{\omega(t)}f^*\chi.
\end{align}
Combining \eqref{C2.1} and \eqref{C2.2}, we can choose a constant $C\ge1$ and a sufficiently large constant $A$ such that
$$\Delta_{\omega(t)}(\log tr_{\omega(t)}f^*\chi-A\varphi(t))\ge tr_{\omega(t)}f^*\chi-C.$$
Then by the maximum principle and Lemma \ref{C0lem}, Lemma \ref{C2lem} follows.
\end{proof}

Before next step, following \cite{To10} we fix a smooth nonnegative function $\varsigma$ on $X$, which vanishes exactly on singular fibers and satisfies $$\varsigma\le1, \sqrt{-1}\partial\varsigma\wedge\bar\partial\varsigma\le Cf^*\chi, -Cf^*\chi\le\sqrt{-1}\partial\bar\partial\varsigma\le Cf^*\chi$$
on $X$ for some constant $C\ge1$.
\begin{lem}\label{C0lem.1}
Set $\bar\varphi_y(t)=\frac{1}{\int_{X_y}(\omega_{0,y})^{n-k}}\int_{X_y}\varphi(t)(\omega_{0,y})^{n-k}$ for $y\in Y\setminus S$. There exists a constants $C\ge1$ such that for all $y\in Y\setminus S$,
\begin{equation}\label{C0lem.1'}
\sup_{X_y\times[2,\infty)}e^t|\varphi(t)-\bar\varphi_y(t)|\le Ce^{C\varsigma^{-C}}.
\end{equation}
\end{lem}
\begin{proof}
Set $\psi(t):=e^t(\varphi(t)-\bar\varphi_y(t))$ and $\omega_{t,y}:=\omega(t)|_{X_y}=e^{-t}\omega_{0,y}+\sqrt{-1}\partial\bar\partial\varphi(t)|_{X_y}$. Then
$$e^t\omega_{t,y}=\omega_{0,b}+\sqrt{-1}\partial\bar\partial\psi(t)|_{X_y}$$
and
\begin{equation}
(\omega_{0,b}+\sqrt{-1}\partial\bar\partial\psi(t)|_{X_y})^{n-k}=e^{(n-k)t}\omega_{t,y}^{n-k}.
\end{equation}
Note that
\begin{align}
\frac{\omega_{t,y}^{n-k}}{\omega_{0,y}^{n-k}}&=\frac{\omega(t)^{n-k}\wedge f^*\chi^k}{\omega_0^{n-k}\wedge f^*\chi^k}\nonumber\\
&=\frac{\omega(t)^{n-k}\wedge f^*\chi^k}{\omega(t)^n}\frac{\omega(t)^n}{\omega_0^{n-k}\wedge f^*\chi^k}\nonumber\\
&\le C(tr_{\omega(t)}f^*\chi)^k\frac{e^{-(n-k)t}}{\varsigma^C}\nonumber\\
&\le C\frac{e^{-(n-k)t}}{\varsigma^C}\nonumber
\end{align}
Hence,
\begin{equation}\label{fib}
(\omega_{0,b}+\sqrt{-1}\partial\bar\partial\psi(t)|_{X_y})^{n-k}=F_y(t)\omega_{0,b}^{n-k}.
\end{equation}
where
\begin{equation}\label{fib1}
F_y(t)\le C\varsigma^{-C}
\end{equation}
for all $y\in Y\setminus S$ and $t\in[1,\infty)$.
Now we separate discussions into two cases.\\

\underline{Case (1): $n-k\ge2$.} For this case, by applying the arguments in \cite[Lemmas 3.2-3.4]{To10} we have
\begin{itemize}
\item[(1)] There exists a uniform constant $C\ge1$ such that for any $y\in Y\setminus S$, $t\in[1,\infty)$ and $u\in C^\infty(X_y)$, we have
\begin{equation}\label{Sob}
\left(\int_{X_y}|u|^{\frac{2(n-k)}{n-k-1}}\omega_{0,y}^{n-k}\right)^{\frac{n-k-1}{n-k}}\le C\int_{X_y}(|\nabla u|^2_{\omega_{0,y}}+|u|^2)\omega_{0,y}^{n-k}.
\end{equation}
\item[(2)] There exists a uniform constant $C\ge1$ such that for any $y\in Y\setminus S$, $t\in[1,\infty)$ and $u\in C^\infty(X_y)$ with $\int_{X_y}u\omega_{0,y}^{n-k}=0$, we have
\begin{equation}\label{Ponc}
\int_{X_y}|u|^{2}\omega_{0,y}^{n-k}\le Ce^{C\varsigma^{-C}}\int_{X_y}|\nabla u|^2_{\omega_{0,y}}\omega_{0,y}^{n-k}.
\end{equation}
\end{itemize}
Now, by combining \eqref{fib}-\eqref{Ponc} and applying Yau's $L^\infty$-estimate, we can conclude \eqref{C0lem.1'} easily.\\

\underline{Case (2): $n-k=1$.} In this case, since the real dimension of a smooth fiber $X_y$ is two (strictly less than three), it seems we can't not apply the above arguments directly. We now make use of the idea in \cite[Corollary 5.2]{ST06} to achieve \eqref{C0lem.1'} in this case.
\par Let $\Delta_{\overline\omega_{0,y}}$ be the Laplacian of $\overline\omega_{0,y}$ on the smooth fiber $X_y$, $G_y(\cdot,\cdot)$ the Green function with respect to $\overline\omega_{0,y}$ on $X_y$ and $A_y:=\inf_{X_y\times X_y}G_y(\cdot,\cdot)$. Then by Green formula for any $x\in X_y$ we have
\begin{equation}\label{green}
\varphi(x)-\frac{1}{V_0}\int_{X_y}\varphi\overline\omega_{0,y}=-\int_{z\in X_y}\Delta_{\overline\omega_{0,y}}\varphi(z)(G_y(x,z)-A_y)\overline\omega_{0,y}(z)
\end{equation}
To estimate \eqref{green}, we first note that its right hand side
\begin{equation}\label{green 1}
\left|-\int_{z\in X_y}\Delta_{\overline\omega_{0,y}}\varphi(z)(G_y(x,z)-A_y)\overline\omega_{0,y}(z)\right|\le\int_{z\in X_y}|\Delta_{\overline\omega_{0,y}}\varphi(z)|(G_y(x,z)-A_y)\overline\omega_{0,y}(z).
\end{equation}
We now collect several claims.\\
\underline{Claim (1)}: there exists a positive constant $C\ge1$ such that for all $y\in Y\setminus S$ and $t\in[1,\infty)$ we have
\begin{equation}
0<e^{-t}+\Delta_{\omega_{0,y}}\varphi\le Ce^{-t}\varsigma^{-C}.
\end{equation}
Claim (1) can be checked as follows:
\begin{align}
0<tr_{\omega_{0,y}}\omega_{t,y}&=\frac{\omega_{t,y}}{\omega_{0,y}}=\frac{\omega(t)\wedge f^*\chi}{\omega_0\wedge f^*\chi}=\frac{\omega(t)\wedge f^*\chi}{\omega(t)^2}\cdot\frac{\omega(t)^2}{\omega_0\wedge f^*\chi}=\frac{1}{2}\left(tr_{\omega(t)}f^*\chi\right)\frac{\omega(t)^2}{\omega_0\wedge f^*\chi}\nonumber\\
&\le Ce^{-t}\frac{\Omega}{\omega_0\wedge f^*\chi}\nonumber\\
&\le Ce^{-t}\varsigma^{-C}\nonumber,
\end{align}
where we have used Lemmas \ref{C0lem'} and \ref{C2lem}. Claim (1) follows.\\

\underline{Claim (2)}: there exists a constant $C\ge1$ such that for any $y\in Y\setminus S$ and $t\in[1,\infty)$ we have
\begin{equation}\label{claim 2}
\frac{\omega_{0,y}}{\overline\omega_{0,y}}\le C\varsigma^{-C}.
\end{equation}
In fact, we have
\begin{align}
\frac{\omega_{0,y}}{\overline\omega_{0,y}}&=\frac{\omega_{0}\wedge f^*\chi}{\overline\omega_{0}\wedge f^*\chi}=\frac{\omega_{0}\wedge f^*\chi}{\Omega}\frac{\Omega}{\overline\omega_{0}\wedge f^*\chi}=2G\frac{\omega_{0}\wedge f^*\chi}{\Omega}\nonumber\\
&\le C\cdot G\nonumber\\
&\le C\varsigma^{-C}\nonumber.
\end{align}
Claim (2) follows.\\

\underline{Claim (3)}: there exists a constant $C\ge1$ such that for any $y\in Y\setminus S$ we have
\begin{align}\label{fiber.diam}
diam(X_y,\overline\omega_{0,y})\le C\varsigma^{-C}.
\end{align}
To see Claim (3), we first recall that, by applying a result of Topping \cite[Theorem 1.1]{Top} (also see \cite[Lemma 3.3]{To10}, whose argument can be applied to our case directly), we can find a constant $C\ge1$ such that for all $y\in Y\setminus S$ there holds
\begin{equation}\label{diam0}
diam(X_y,\omega_{0,y})\le C.
\end{equation}
Secondly, we apply a similar argument in Claim (2) to see
\begin{align}\label{diam1}
\frac{\overline\omega_{0,y}}{\omega_{0,y}}&=\frac{\overline\omega_{0}\wedge f^*\chi}{\omega_{0}\wedge f^*\chi}=\frac{\Omega}{\omega_{0}\wedge f^*\chi}\frac{\overline\omega_{0}\wedge f^*\chi}{\Omega}=\frac{1}{2G}\frac{\Omega}{\omega_{0}\wedge f^*\chi}\nonumber\\
&\le C\varsigma^{-C},
\end{align}
where we have used a positive lower bound of $G$ contained in \eqref{Lp}. Combining \eqref{diam0} and \eqref{diam1}, Claim (3) follows.\\

Now we can complete the proof. Since we have
$$Ric(\overline\omega_{0,y})=\omega_{0,y}>0,$$
i.e. Ricci curvature of $\overline\omega_{0,y}$ is uniformly bounded from below by zero, and the volume of $(X_y,\overline\omega_{0,y})$ is a positive constant $V_0$, we can apply a result in \cite[Section 1.1 in Chapter 3]{Si} (note that in \cite[Appendix A of Section 3]{Si} a proof for manifold with real dimension greater than or equals to three is provided; for real 2-dimensional case, this result can be checked by using $L^1$-Sobolev inequality and lower bound of isoperimetric constant contained in \cite[Proposition 4 and Theorem 13]{Cr}) to find a uniform positive constant $\gamma$ such that for all $y\in Y\setminus S$ we have
\begin{align}\label{green 2}
G_y(\cdot,\cdot)&\ge-\gamma\frac{diam(X_y,\overline\omega_{0,y})^2}{V_0}.
\end{align}
Plugging \eqref{fiber.diam} into \eqref{green 2} gives
\begin{align}
G_y(\cdot,\cdot)&\ge-C\varsigma^{-C}\nonumber
\end{align}
and hence
\begin{align}\label{claim 4}
A_y\ge-C\varsigma^{-C}.
\end{align}
On the other hand, combining Claims (1) and (2) we have
\begin{align}\label{claim 5}
|\Delta_{\overline\omega_{0,y}}\varphi|&=|\Delta_{\omega_{0,y}}\varphi|\frac{\omega_{0,y}}{\overline\omega_{0,y}}\nonumber\\
&\le Ce^{-t}\varsigma^{-C}.
\end{align}
Plugging \eqref{claim 4} and \eqref{claim 5} into \eqref{green 1} we find that
$$|\varphi-\frac{1}{V_0}\int_{X_y}\varphi\overline\omega_{0,y}|\le Ce^{-t}\varsigma^{-C}$$
on $X_y$ and hence
$$\sup_{X_y}\varphi-\inf_{X_y}\varphi\le Ce^{-t}\varsigma^{-C},$$
which implies
$$\sup_{X_y\times[2,\infty)}e^{t}|\varphi(t)-\bar\varphi_y(t)|\le C\varsigma^{-C}.$$
\par Lemma \ref{C0lem.1} is proved.
\end{proof}

Define a smooth function $\bar\varphi(t)$ on $Y\setminus S\times[1,\infty)$ by $\bar\varphi(y,t):=\bar\varphi_y(t)$.

\begin{lem}\label{C2lem.2}\cite[Lemma 5.9]{ST06}
There exists a uniform constant $C\ge1$ such that on $Y\times[1,\infty)$ we have
\begin{equation}\label{C2lem.2'}
\Delta_{\omega(t)}(e^{t}(\varphi(t)-\bar\varphi(t)))\le-tr_{\omega(t)}\omega_0+Ce^{t}+C\varsigma^{-C}.
\end{equation}
\end{lem}
\begin{proof}
For a proof, see \cite[Lemma 5.9]{ST06}.
\end{proof}

\begin{lem}\label{C2lem.3}
There exists a uniform constant $C\ge1$ such that on $Y\times[1,\infty)$ we have
\begin{equation}\label{C2lem.3'}
tr_{\omega(t)}(e^{-t}\omega_0)\le e^{Ce^{C\varsigma^{-C}}}.
\end{equation}
\end{lem}
\begin{proof}
To prove this lemma, we simply modify arguments in \cite[Theorem 5.2]{ST06}. Firstly, by Schwarz lemma argument (see e.g. \cite{Y78}) we have a constant $C\ge1$ such that for $t\in[1,\infty)$,
\begin{equation}\label{C2lem.3.1}
\Delta_{\omega(t)}\log tr_{\omega(t)}(e^{-t}\omega_0)\ge-Ctr _{\omega(t)}\omega-C.
\end{equation}
Combining \eqref{C2lem.2'} and \eqref{C2lem.3.1}, we know for some constants $C\ge1$ and $A\ge1$ and all $t\in[1,\infty)$,
\begin{equation}\label{C2lem.3.2}
\Delta_{\omega(t)}(\log tr_{\omega(t)}(e^{-t}\omega_0)-Ae^t(\varphi(t)-\bar\varphi(t)))\ge tr_{\omega(t)}\omega_0-Ce^t-C\varsigma^{-C}.
\end{equation}
Choose a positive constant $C_1$ such that $C_1-2$ satisfies Lemma \ref{C0lem.1}. Set
$$K=\log tr_{\omega(t)}(e^{-t}\omega_0)-Ae^t(\varphi(t)-\bar\varphi(t))$$
and $H=e^{-C_1\varsigma^{-C_1}}K$. Compute
\begin{align}\label{a0}
\Delta_{\omega(t)}H&=e^{-C_1\varsigma^{-C_1}}\Delta_{\omega(t)}K+K\Delta_{\omega(t)}(e^{-C_1\varsigma^{-C_1}})+2Re\langle\nabla(e^{-C_1\varsigma^{-C_1}}),\bar\nabla K\rangle_{\omega(t)}\nonumber\\
&=:I+II+III.
\end{align}
Firstly, using \eqref{C2lem.3.2}, we see that
\begin{align}\label{a}
I=e^{-C_1\varsigma^{-C_1}}\Delta_{\omega(t)}K&\ge e^{-C_1\varsigma^{-C_1}}(tr_{\omega(t)}\omega_0-Ce^t-C\varsigma^{-C})\nonumber\\
&\ge  e^{-C_1\varsigma^{-C_1}}tr_{\omega(t)}\omega_0-Ce^t.
\end{align}
Secondly, using
$$|\Delta_{\omega(t)}(e^{-C_1\varsigma^{-C_1}})|\le\frac{C_2e^{-C_1\varsigma^{-C_1}}}{\varsigma^{C_1+2}}$$
for some constant $C_2\ge1$, we see that
\begin{align}\label{b}
II=K\Delta_{\omega(t)}(e^{-C_1\varsigma^{-C_1}})&\ge-\frac{C_2e^{-C_1\varsigma^{-C_1}}}{\varsigma^{C_1+2}}|\log tr_{\omega(t)}\omega_0|-Ct\nonumber\\
&=-\frac{C_2e^{-C_1\varsigma^{-C_1}}}{\varsigma^{C_1+2}}|2\log \sqrt{tr_{\omega(t)}\omega_0}|-Ct\nonumber\\
&\ge-\frac{C_2e^{-C_1\varsigma^{-C_1}}}{\varsigma^{C_1+2}}\left(\frac{\varsigma^{C_1+2}tr_{\omega(t)}\omega_0}{4C_2}+4C_2\varsigma^{-(C_1+2)}\right)-Ct\nonumber\\
&\ge-\frac{1}{4}e^{-C_1\varsigma^{-C_1}}tr_{\omega(t)}\omega_0-Ct.
\end{align}
On the other hand, using
$$|\nabla\varsigma|_{\omega(t)}^2\le \frac{C_3e^{-2C_1\varsigma^{-C_1}}}{\varsigma^{2C_1+2}}$$
for some constant $C_3\ge1$, we see that
\begin{align}\label{c}
III&=2Re\langle\nabla(e^{-C_1\varsigma^{-C_1}}),\bar\nabla K\rangle_{\omega(t)}\nonumber\\
&=2Re\langle\nabla(e^{-C_1\varsigma^{-C_1}}),\bar\nabla (\frac{H}{e^{-C_1\varsigma^{-C_1}}})\rangle_{\omega(t)}\nonumber\\
&=2e^{C_1\varsigma^{-C_1}}Re\langle\nabla(e^{-C_1\varsigma^{-C_1}}),\bar\nabla H\rangle_{\omega(t)}-\frac{2H}{e^{-2C_1\varsigma^{-C_1}}}|\nabla(e^{-C_1\varsigma^{-C_1}})|^2_{\omega(t)}\nonumber\\
&\ge2e^{C_1\varsigma^{-C_1}}Re\langle\nabla(e^{-C_1\varsigma^{-C_1}}),\bar\nabla H\rangle_{\omega(t)}-2C_3e^{-C_1\varsigma^{-C_1}}(|\log tr_{\omega(t)}(e^{-t}\omega_0)|+|Ae^t(\varphi(t)-\bar\varphi(t))|)\varsigma^{-2(C_1+2)}\nonumber\\
&\ge2e^{C_1\varsigma^{-C_1}}Re\langle\nabla(e^{-C_1\varsigma^{-C_1}}),\bar\nabla H\rangle_{\omega(t)}-2C_3\frac{e^{-C_1\varsigma^{-C_1}}}{\varsigma^{2C_1+2}}|\log tr_{\omega(t)}\omega_0|-Ct\nonumber\\
&\ge2e^{C_1\varsigma^{-C_1}}Re\langle\nabla(e^{-C_1\varsigma^{-C_1}}),\bar\nabla H\rangle_{\omega(t)}-2C_3\frac{e^{-C_1\varsigma^{-C_1}}}{\varsigma^{2C_1+2}}\left(\frac{\varsigma^{2C_1+2}tr_{\omega(t)}\omega_0}{8C_3}+8C_3\varsigma^{-(2C_1+2)}\right)-Ct\nonumber\\
&\ge2e^{C_1\varsigma^{-C_1}}Re\langle\nabla(e^{-C_1\varsigma^{-C_1}}),\bar\nabla H\rangle_{\omega(t)}-\frac{1}{4}e^{-C_1\varsigma^{-C_1}}tr_{\omega(t)}\omega_0-Ct.
\end{align}
Now we plug \eqref{a}, \eqref{b} and \eqref{c} into \eqref{a0} and obtain
\begin{equation}\label{d}
\Delta_{\omega(t)}H\ge\frac{1}{2}e^{-C_1\varsigma^{-C_1}}tr_{\omega(t)}\omega_0+2e^{C_1\varsigma^{-C_1}}Re\langle\nabla(e^{-C_1\varsigma^{-C_1}}),\bar\nabla H\rangle_{\omega(t)}-Ce^{t}.
\end{equation}
For any $t\in[1,\infty)$, let $x_t$ be a maximal point of $H(t)$. If $x_t\in Y\setminus Y_{reg}$, then $H(x_t,t)\le C$ for some uniform constant $C\ge1$; if $x_t\in Y_{reg}$, we apply the maximal principle in \eqref{d} to see that
$$tr_{\omega(t)}(e^{-t}\omega_0)\le Ce^{C_1\varsigma^{-C_1}}$$
at $x_t$ and hence
$$H(x_t,t)\le C$$
for some uniform constant $C\ge1$. In conclusion, there exists a uniform constant $C\ge1$ such that for all $t\in[1,\infty)$,
$$H(t)\le C,$$
from which we see that
$$tr_{\omega(t)}(e^{-t}\omega_0)\le e^{Ce^{C\varsigma^{-C}}}$$
for some uniform constant $C\ge1$.
\par Lemma \ref{C2lem.3} is proved.
\end{proof}
\begin{lem}\label{C2lem.4}
There exists a constant $C\ge$ such that for all $t\in[1,\infty)$,
\begin{equation}\label{e}
e^{-Ce^{C\varsigma^{-C}}}\omega_t\le\omega(t)\le e^{Ce^{C\varsigma^{-C}}}\omega_t.
\end{equation}
\end{lem}
\begin{proof}
The left hand side of \eqref{e} follows by combining Lemmas \ref{C2lem} and \ref{C2lem.3}. For the right hand side,
\begin{align}
tr_{\omega_t}\omega(t)&\le(n-1)!(tr_{\omega(t)}\omega_t)^{n-1}\frac{\omega(t)^n}{\omega_t^n}\le e^{C(n-1)e^{C\varsigma^{-C}}}\frac{C}{\varsigma^{C}}\le e^{C'e^{C'\varsigma^{-C'}}}\nonumber
\end{align}
for some uniform constant $C'\ge1$, from which the right hand side of \eqref{e} follows.
\par Lemma \ref{C2lem.4} is proved.
\end{proof}

We arrive at the main result in this section.
\begin{thm}\label{L1conv}
As $t\to\infty$, $\varphi(t)\to f^*\psi$ in $L^1(X,\Omega)$- and $C^{1,\alpha}(X_{reg},\omega_0)$-topology for any $\alpha\in(0,1)$. Here $\psi$ is the unique solution to \eqref{limit}. Consequently, as $t\to\infty$, $\omega(t)\to f^*\omega_Y$ in the current sense on $X$.
\end{thm}
\begin{proof}
Using the same arguments in \cite[Lemmas 2.7, 2.8]{ZyZz} (also see \cite{To10}), for any time sequence $t_j\to\infty$ we can find a subsequence, still denote by $t_j$, and a $\tilde\psi\in PSH(Y,\chi)\cap L^\infty(Y)$ such that $\varphi(t_j)\to f^*\tilde\psi$ in $L^1(X,\Omega)$-topology and $\tilde\psi$ satisfies \eqref{limit} on $Y\setminus S$, i.e. for any $K\subset\subset Y\setminus S$ and any given $\varrho\in C^\infty_0(K)$, there holds
\begin{equation}\label{gkeeq1}
\int_{Y}\varrho(\chi+\sqrt{-1}\partial\bar\partial\tilde\psi)^k=\int_{Y}\varrho Ge^{\tilde\psi}\chi^k.
\end{equation}
To see that $\tilde\psi$ satisfies \eqref{limit} on $S$, we need to recall some arguments in \cite{EGZ,ST12} on how to solve \eqref{limit} on $Y$. Firstly, we choose a resolution of singularities of $Y$:
$$\pi:\hat Y\to Y,$$
namely, $\hat Y$ is nonsingular, $\pi(\hat Y)=Y$ and $\pi:\hat Y\setminus \pi^{-1}(S)\to Y\setminus S$ is biholomorphic. Then $\pi^*\chi$ is a smooth semi-positive closed real $(1,1)$-form on $\hat Y$, which is big in the sense that $\int_{\hat Y}(\pi^*\chi)^k>0$. We pullback the equation \eqref{limit} to $\hat Y$:
\begin{equation}\label{limit1}
(\pi^*\chi+\sqrt{-1}\partial\bar\partial\hat\psi)^k=e^{\hat\psi}\pi^*G(\pi^*\chi)^k.
\end{equation}
Applying \cite[Theorem 4.1]{EGZ} or \cite[Theorem 3.2]{ST12} gives a unique $\hat \psi\in PSH(\hat Y,\pi^*\chi)\cap L^\infty(\hat Y)$ solving \eqref{limit1} on $\hat Y$. Then $\hat \psi$ is constant along every fiber of $\pi$ and hence decent to the unique bounded solution $\psi$ on $Y$ solving \eqref{limit} (see e.g. \cite[Theorem 6.3]{EGZ}).

\par Let's be back to our proof. Note that the pullback $\pi^*\tilde\psi$ on $\hat Y$ of $\tilde\psi$ obviously satisfies \eqref{limit1} on $\hat Y\setminus \pi^{-1}(S)$. Moreover, since $\pi^*\tilde\psi$ is a bounded function on $\hat Y$ and $\pi^{-1}(S)$ is a proper subvariety of $\hat Y$, we know $(\pi^*\chi+\sqrt{-1}\partial\bar\partial\pi^*\tilde\psi)^k$ takes no mass on $\pi^{-1}(S)$ (see e.g. \cite{K05}). On the other hand, using the fact that $\pi^*G\in L^{1+\epsilon}(\hat Y)$ and H\"{o}lder inequality, one easily sees that $e^{\pi^*\tilde\psi}\pi^*G(\pi^*\chi)^k$ also takes no mass on $\pi^{-1}(S)$. In conclusion, $\pi^*\tilde\psi$ is also a bounded solution to \eqref{limit1} on $\hat Y$. By uniqueness we have $\hat\psi=\pi^*\tilde\psi$ and hence $\psi=\tilde\psi$. Therefore, as $t\to\infty$, $\varphi(t)\to f^*\psi$ in $L^1(X,\Omega)$-topology without passing to a subsequence.
\par Moreover, for any $K\subset\subset X_{reg}$, by Lemma \ref{C2lem.4} we have a positive constant $C$ such that
$$\sup_{K\times[1,\infty)}|\Delta_{\omega_0}\varphi(t)|\le C.$$
Therefore, given the above $L^1$-convergence, we conclude $C^{1,\alpha}$-convergence by standard elliptic equation theory.

\par Theorem \ref{L1conv} is proved.
\end{proof}

\section{Uniform convergence away from singular fibers}\label{pf}
In this section, we will give a proof for second part of Theorem \ref{t1}, i.e. uniform convergence of metric away from singular fibers, by using the strategy developed in \cite{TWY}. To this end, we need more estimates. For convenience, we will use the following notation: \emph{$G(t),G_i(t),i=1,2,\ldots$, will always denote a positive function of $t$ which converge to $0$ as $t\to\infty$.}

\par Let's begin by the following
\begin{lem}\label{lemconv0}\cite[Lemma 4.3]{TWY}
There exist a constant $C\ge1$ and a positive function $G(t)$ with $G(t)\to0$ as $t\to\infty$ such that
$$\sup_Y e^{-Ce^{C\varsigma^{-C}}}|\varphi(t)-f^*\psi|\le G(t).$$
\end{lem}
\begin{proof}
For a proof, see \cite[Lemma 4.3]{TWY}.
\end{proof}

\begin{lem}\label{lemconv1}
There exist a constant $C\ge1$ and a positive function $G(t)$ with $G(t)\to0$ as $t\to\infty$ such that
$$\sup_Y e^{-Ce^{C\varsigma^{-C}}}|\dot\varphi(t)|\le G(t).$$
\end{lem}
\begin{proof}
We choose a sufficiently large $C$ such that $E=e^{-Ce^{C\varsigma^{-C}}}$ satisfies Lemma \ref{lemconv0}.
We begin by collecting two useful identities as follows, which can be checked by direct computations. Firstly, \eqref{lt1} can be rewritten as follows:
\begin{equation}\label{b1}
\varphi(t)=(1-e^{-t})\log\frac{e^{(n-k)t}\omega(t)^n}{\Omega}.
\end{equation}
Now by taking time derivative of \eqref{b1} we have
\begin{equation}\label{b2}
(1-e^{-t})\Delta_{\omega(t)}(\varphi(t)+\dot\varphi(t))=\dot\varphi(t)-(1-e^{-t})(tr_{\omega(t)}f^*\chi-k)-e^{-t}\log\frac{e^{(n-k)t}\omega(t)^n}{\Omega}
\end{equation}
and
\begin{align}\label{b3}
&\Delta_{\omega(t)}((1-e^{-t})\ddot\varphi(t)+2e^{-t}\dot\varphi(t)-(1-3e^{-t})\varphi(t))\nonumber\\
&=\ddot\varphi(t)+e^{-t}\log\frac{e^{(n-k)t}\omega(t)^n}{\Omega}+(1-3e^{-t})tr_{\omega(t)}f^*\chi+(n+2k)e^{-t}+(1-e^{-t})|\dot\omega(t)|_{\omega(t)}^2-n.
\end{align}
By applying the maximum principle in \eqref{b2} and using Lemmas \ref{C0lem}, \ref{C0lem'} and \ref{C2lem} one finds that
\begin{equation}\label{b4}
\sup_{X\times[1,\infty)}|\dot\varphi(t)|\le C.
\end{equation}
Similarly, we apply the maximum principle in \eqref{b3} and see that
\begin{equation}\label{b5}
\sup_{X\times[1,\infty)}\ddot\varphi(t)\le C.
\end{equation}
Indeed, for any $t\in[1,\infty)$, at a maximal point $x_t$ of $(1-e^{-t})\ddot\varphi(t)+2e^{-t}\dot\varphi(t)-(1-3e^{-t})\varphi(t)$, by \eqref{b3} we have
$$\ddot\varphi(t)(x_t)\le C$$
for some uniform constant $C\ge1$. But $2e^{-t}\dot\varphi(t)-(1-3e^{-t})\varphi(t)$ is uniformly bounded by Lemma \ref{C0lem} and \eqref{b5}. Therefore,
$$((1-e^{-t})\ddot\varphi(t)+2e^{-t}\dot\varphi(t)-(1-3e^{-t})\varphi(t))(x_t)\le C$$
for some uniform constant $C\ge1$ and, using again that $2e^{-t}\dot\varphi(t)-(1-3e^{-t})\varphi(t)$ is uniformly bounded, \eqref{b5} is checked. Now we can complete the proof by an easy argument (see \cite[Lemma 4.6]{TWY}). Assume this lemma fails, then we can find a constant $\delta_0>0$, and sequences $t_j\to\infty$, $x_j\in X$ such that
$$E(x_j)\dot\varphi(t_j)(x_j)\ge\delta_0,$$
which in particular implies that $x_j\in X_{reg}$. On the other hand, by \eqref{b5} we see
$$\partial_t(E\dot\varphi(t))=E\ddot\varphi(t)\le C.$$
So $E(x_j)\dot\varphi(t,x_j)\ge\frac{\delta_0}{2}$ on $t\in[t_j,t_j+\frac{\delta_0}{2C}]$ and hence, by integrating in $t$,
$$E(x_j)(\varphi-f^*\psi)(t_j+\frac{\delta_0}{C},x_j)\ge E(x_j)(\varphi-f^*\psi)(t_j,x_j)+\frac{\delta_0^2}{2C},$$
which implies
\begin{equation}\label{b6}
G(t_j+\frac{\delta_0}{C})\ge G(t_j)+\frac{\delta_0^2}{2C},
\end{equation}
where $G(t)$ is the function in Lemma \ref{lemconv0}. As $t_j\to\infty$, \eqref{b6} is impossible to hold. Therefore,
$$\sup_XE\dot\varphi(t)\le G_1(t)$$
for some positive function $G_1(t)$ with $G_1(t)\to0$ as $t\to\infty$. Similarly, we get
$$\inf_XE\dot\varphi(t)\ge -G_1(t).$$
Lemma \ref{lemconv1} is proved.
\end{proof}

Combining Lemmas \ref{lemconv0} and \ref{lemconv1}, we conclude
\begin{lem}\label{lemconv}
There exist a constant $C\ge1$ and a positive function $G(t)$ with $G(t)\to0$ as $t\to\infty$ such that
$$\sup_Y e^{-Ce^{C\varsigma^{-C}}}|\varphi(t)+\dot\varphi(t)-f^*\psi|\le G(t).$$
\end{lem}

\begin{lem}\label{lemconv2}
Let $G(t)$ be the same function as in Lemma \ref{lemconv}. There exists a constant $C\ge1$ such that
$$\sup_Y e^{-Ce^{C\varsigma^{-C}}}(tr_{\omega(t)}f^*\omega_{Y}-k)\le C\sqrt{G(t)}.$$
\end{lem}
\begin{proof}
For convenience, we present a proof by following \cite[Lemma 4.7]{TWY}. Choose $E=e^{-Ce^{C\varsigma^{-C}}}$ satisfies Lemma \ref{lemconv}. We also choose a sufficiently large constant $C_1\ge$ such that $E_1=e^{-C_1e^{C_1\varsigma^{-C_1}}}$ satisfies $\frac{|\partial E_1|_{\omega(t)}^2}{E_1}\le CE$ and $|\Delta_{\omega(t)} E_1|\le CE$. On the one hand, we have
\begin{align}
&\Delta_{\omega(t)}(E_1(\dot\varphi(t)+\varphi(t)-f^*\psi))\nonumber\\
&=E_1\Delta_{\omega(t)}(\dot\varphi(t)+\varphi(t)-f^*\psi)+(\dot\varphi(t)+\varphi(t)-f^*\psi)\Delta_{\omega(t)}E_1+2Re\langle\partial E_1,\bar\partial(\dot\varphi(t)+\varphi(t)-f^*\psi)\rangle_{\omega(t)}\nonumber\\
&\le CG(t)+CE_1(-tr_{\omega(t)}f^*\omega_{Y}+k+\dot\varphi+e^{-t})+2Re\langle\partial E_1,\bar\partial(\dot\varphi(t)+\varphi(t)-f^*\psi)\rangle_{\omega(t)}\nonumber\\
&\le -CE_1(tr_{\omega(t)}f^*\omega_{Y}-k)+CG(t)+2Re\langle\partial E_1,\bar\partial(\dot\varphi(t)+\varphi(t)-f^*\psi)\rangle_{\omega(t)}\nonumber,
\end{align}
where we have used \eqref{b2} and Lemma \ref{lemconv1} and assumed without loss of generality that $e^{-t}\le CG(t)$ on $[1,\infty)$. Hence,
\begin{align}\label{b7}
&\Delta_{\omega(t)}\left(\frac{E_1(\dot\varphi(t)+\varphi(t)-f^*\psi)}{\sqrt{G(t)}}\right)\nonumber\\
&\le-C\frac{E_1(tr_{\omega(t)}f^*\omega_{Y}-k)}{\sqrt{G(t)}}+C\sqrt{G(t)}+2Re\langle\partial E_1,\frac{\bar\partial(\dot\varphi(t)+\varphi(t)-f^*\psi)}{\sqrt{G(t)}}\rangle_{\omega(t)}.
\end{align}
On the other hand, using \cite[Section 3]{ST12}, we know $\omega_{Y}\le C\varsigma^{-C}\chi$ on $Y\setminus S$ and the bisectional curvature of $\omega_{Y}$ on $Y\setminus S$ has an upper bound of the form $\varsigma^{-C'}$ for some constant $C'\ge1$, so by a Schwarz lemma argument we see
\begin{equation}
\Delta_{\omega(t)}tr_{\omega(t)}f^*\omega_{Y}\ge -C'tr_{\omega(t)}f^*\omega_{Y}-C'\varsigma^{-C'}(tr_{\omega(t)}f^*\omega_{Y})^2\ge-C\varsigma^{-C}
\end{equation}
Therefore, we have
\begin{align}\label{b8}
&\Delta_{\omega(t)}(E_1(tr_{\omega(t)}f^*\omega_{Y}-k))\nonumber\\
&=(tr_{\omega(t)}f^*\omega_{Y}-k)\Delta_{\omega(t)}E_1+E_1\Delta_{\omega(t)}(tr_{\omega(t)}f^*\omega_{Y}-k)+2Re\langle\partial E_1,\bar\partial(tr_{\omega(t)}f^*\omega_{Y}-k)\rangle\nonumber\\
&\ge-C E_1\varsigma^{-C}+2Re\langle\partial E_1,\bar\partial(tr_{\omega(t)}f^*\omega_{Y}-k)\rangle\nonumber\\
&\ge2Re\langle\partial E_1,\bar\partial(tr_{\omega(t)}f^*\omega_{Y}-k)\rangle-C.
\end{align}
Now we set $H_1=E_1(tr_{\omega(t)}f^*\omega_{Y}-k)-\frac{E_1(\dot\varphi(t)+\varphi(t)-f^*\psi)}{\sqrt{G(t)}}$. Note that $H_1$ is uniformly bounded and we want to show that it converges to $0$ uniformly by applying the maximum principle. By \eqref{b7} and \eqref{b8} we see that
\begin{align}\label{b9}
\Delta_{\omega(t)}H_1&\ge C\frac{E_1(tr_{\omega(t)}f^*\omega_{Y}-k)}{\sqrt{G(t)}}-C\sqrt{G(t)}-C\nonumber\\
&-2Re\langle\partial E_1,\frac{\bar\partial(\dot\varphi(t)+\varphi(t)-f^*\psi)}{\sqrt{G(t)}}\rangle_{\omega(t)}+2Re\langle\partial E_1,\bar\partial(tr_{\omega(t)}f^*\omega_{Y}-k)\rangle_{\omega(t)}\nonumber\\
&=C\frac{E_1(tr_{\omega(t)}f^*\omega_{Y}-k)}{\sqrt{G(t)}}-C+2Re\langle\partial E_1,\bar\partial\left(\frac{H_1}{E_1}\right)\rangle_{\omega(t)}\nonumber\\
&=C\frac{E_1(tr_{\omega(t)}f^*\omega_{Y}-k)}{\sqrt{G(t)}}-C+\frac{2}{E_1}Re\langle\partial E_1,\bar\partial H_1\rangle_{\omega(t)}-\frac{2H_1}{E_1^2}|\partial E_1|_{\omega(t)}^2\nonumber\\
&\ge C\frac{E_1(tr_{\omega(t)}f^*\omega_{Y}-k)}{\sqrt{G(t)}}+\frac{2}{E_1}Re\langle\partial E_1,\bar\partial H_1\rangle_{\omega(t)}-C
\end{align}
Then by applying the maximum principle to \eqref{b9}, we easily complete the proof of this lemma.
\par Lemma \ref{lemconv2} is proved.
\end{proof}

\begin{lem}\label{lemconv3}
For any $K\subset Y\setminus S$, there exists a constant $C=C_K\ge1$ such that for all $y\in K$ we have
\begin{equation}\label{c1}
|e^t\omega(t)|_{X_y}|_{C^1(X_y,\omega_{0,y})}\le C.
\end{equation}
\end{lem}
\begin{proof}
We will make use of some arguments in \cite[Theorem 1.1]{ToZy15}. For any given $y\in K$ and $x\in X_y$, we choose a local chart $(U,w^1,\ldots,w^{n-k},y^1,\ldots,y^k)$ on $X$ centered at $x$ and local chart $(V,y^1,\ldots,y^k)$ on $Y$ centered at $y$ such that in these local charts $f$ is given by $(w^1,\ldots,w^{n-k},y^1,\ldots,y^k)\mapsto(y^1,\ldots,y^k)$. We also assume that $U=\{(w^1,\ldots,w^{n-k},y^1,\ldots,y^k)\in\mathbb{C}^n||w^i|<1,|y^\alpha|<1,i=1,\ldots,n-k,\alpha=1,\ldots,k\}$ be the polydisc in $\mathbb{C}^n$. For each $t\ge0$, let $B_{e^{\frac{t}{2}}}\subset\mathbb{C}^k$ be the polydisc in $\mathbb{C}^n$ centered at $0$ and $D$ be the unit polydisc in $\mathbb{C}^{n-k}$.
\par Consider the maps $F_t:D\times B_{e^{\frac{t}{2}}}\to U$, $F_t(w,y)=(w,ye^{-\frac{t}{2}})$. Note that $F_t$ is the identity when restricting to $D\times\{0\}$. On $U$ we can write

\begin{align}
\omega_0(w,y)=&\sqrt{-1}\left(\sum_{i,j=1}^{n-k}g_{i\bar j}(w,y)dw^i\wedge d\bar w^j+\sum_{\alpha,\beta=1}^{k}g_{\alpha\bar\beta}(w,y)dy^\alpha\wedge d\bar y^{\beta}\right)\nonumber\\
&+2Re\sqrt{-1}\sum_{i=1}^{n-k}\sum_{\alpha=1}^{k}g_{i\bar\alpha}(w,y)dw^i\wedge d\bar y^{\alpha}
\end{align}
and then
\begin{align}
F_t^*\omega_0(w,b)=&\sqrt{-1}\sum_{i,j=1}^{n-k}g_{i\bar j}(w,ye^{-\frac{t}{2}})dw^i\wedge d\bar w^j+e^{-t}\sqrt{-1}\sum_{\alpha,\beta=1}^kg_{\alpha\bar\beta}(w,ye^{-\frac{t}{2}})dy^\alpha\wedge d\bar y^{\beta}\nonumber\\
&+2e^{-\frac{t}{2}}Re\sqrt{-1}\sum_{i=1}^{n-k}\sum_{\alpha=1}^kg_{i\bar\alpha}(w,ye^{-\frac{t}{2}})dw^i\wedge d\bar y^\alpha\nonumber.
\end{align}
Note that $F_t^*\omega_0$ converges on any compact subsets of $D\times\mathbb{C}^k$, as $t\to\infty$, to a smooth semi-positive real $(1,1)$-form $\eta=\sqrt{-1}\sum_{i,j=1}^{n-k}g_{i\bar j}(w,0)dw^i\wedge d\bar w^j$, which is strictly positive in fiber direction. Also note that
$$e^tF_t^*f^*\chi=\sqrt{-1}\sum_{\alpha,\beta=1}^k\chi_{\alpha\bar\beta}(w,ye^{-\frac{t}{2}})dy^\alpha\wedge d\bar y^{\beta}$$
converges on any compact subsets of $D\times\mathbb{C}^k$, as $t\to\infty$, to a semi-positive real $(1,1)$-form $\eta'=\sqrt{-1}\sum_{\alpha,\beta=1}^k\chi_{\alpha\bar\beta}(0)dy^\alpha\wedge d\bar y^{\beta}$, which is strictly positive in base direction. In conclusion,
$e^tF_t^*\omega_t=(1-e^{-t})e^tF_t^*f^*\chi+F_t^*\omega_0$ converges on any compact subsets of $D\times\mathbb{C}^k$, as $t\to\infty$, to a smooth K\"ahler metric $\eta+\eta'$, which is equivalent to the standard Euclidean metric $\omega_E$ on $D\times\mathbb{C}^k$. Therefore, for any $K'\subset\subset D\times\mathbb{C}^k$, by Lemma \ref{C2lem.4} one can find a constant $C=C_{K'}\ge1$ with
\begin{equation}\label{c0}
C^{-1}\omega_E\le e^tF_t^*\omega(t)\le C\omega_E
\end{equation}
on $K$, where $\omega_E$ is the Euclidean metric. On the other hand, we have
\begin{align}\label{c5}
Ric(e^tF_t^*\omega(t))=\frac{1}{1-e^{-t}}F_t^*\omega_0-\frac{e^{-t}}{1-e^{-t}}(e^tF_t^*\omega(t)).
\end{align}
With all the above preparations, now we are able to derive a $C^1$-estimate for $e^tF_t^*\omega(t)$, so called Calabi-type estimate, by modifying \cite{ShWe}. For simplicity, we write $\tilde\omega=\tilde\omega(t)=e^tF_t^*\omega(t)$. Firstly, we fix a slightly large $K''\subset\subset D\times\mathbb{C}^k$ containing $K'$ in its interior, so that, for some uniform constant $C=C_{K''}$,
\begin{equation}\label{c3}
C^{-1}\omega_E\le\tilde\omega(t)\le C\omega_E.
\end{equation}
Secondly, we choose a smooth cut-off function $\varrho$ supported in $K''$ with $\varrho\equiv1$ on $K'$. By \eqref{c3}, we can choose a $\varrho$ satisfying
$$|\partial\varrho|_{\tilde\omega(t)}+|\Delta_{\tilde\omega(t)}\varrho|\le C$$
for some uniform constant $C\ge1$. Let $\nabla^E,\Gamma^E$ be the connection and Christoffel symbols of $\omega_E$ and $\tilde\nabla,\tilde\Gamma$ those of $\tilde\omega$. Define $T$ to be the tensor that is difference of the Christoffel symbols of $\tilde\omega$ and $\omega_E$ (in fact, $\Gamma^E\equiv0$ and hence $T=\tilde\Gamma$) and $S=|T|_{\tilde\omega}^2$. Easily, we have
\begin{align}\label{c4}
S=|T|_{\tilde\omega}^2=|\nabla^E\tilde\omega|_{\tilde\omega}^2=\tilde g^{\bar li}\tilde g^{\bar qj}\tilde g^{\bar kp}\partial_i\tilde g_{j\bar k}\partial_{\bar l}\tilde g_{p\bar q}.
\end{align}
Using commutation formula, one obtains
\begin{align}\label{c6}
\Delta_{\tilde\omega}S=&|\tilde\nabla T|_{\tilde \omega}^2+|\overline{\tilde{\nabla}} T|_{\tilde\omega}^2+2Re(\tilde g^{\bar ji}\tilde g^{\bar qp}\tilde g_{k\bar l}\Delta_{\tilde\omega}T_{ip}^k\overline T_{jq}^l)\nonumber\\
&+\tilde g^{\bar bi}\tilde g^{\bar ja}\tilde g^{\bar qp}\tilde g_{k\bar l}T_{ip}^k\overline T_{jq}^l\tilde R_{a\bar b}+\tilde g^{\bar bp}\tilde g^{\bar qa}\tilde g^{\bar ji}\tilde g_{j\bar l}T_{ip}^k\overline T_{jq}^l\tilde R_{a\bar b}-\tilde g^{\bar ji}\tilde g^{\bar qp}T_{ip}^k\overline T_{jq}^l\tilde R_{k\bar l}.
\end{align}
By plugging \eqref{c5} into the last three terms of \eqref{c6}, and noting \eqref{c3} and the fact that, if we write $\bar\omega=F_t^*\omega_0$, then $0\le\bar\omega\le C\omega_{E}$ on $K''\times[1,\infty)$ for some uniform constant $C\ge1$ we have
\begin{align}\label{c7}
\Delta_{\tilde\omega}S\ge|\tilde\nabla T|_{\tilde \omega}^2+|\overline{\tilde{\nabla}} T|_{\tilde\omega}^2+2Re(\tilde g^{\bar ji}\tilde g^{\bar qp}\tilde g_{k\bar l}(\Delta_{\tilde\omega}T_{ip}^k)\overline T_{jq}^l)-CS.
\end{align}
Moreover, by the symmetry of curvature tensor and Bianchi identities we see
\begin{align}\label{c8}
\Delta_{\tilde\omega}T_{ip}^k&=\tilde g^{\bar ba}\tilde\nabla_a\tilde\nabla_{\bar b}T_{ip}^k=-\tilde g^{\bar dk}\tilde\nabla_i\tilde R_{p\bar d}=-\tilde g^{\bar dk}\tilde\nabla_i(\frac{1}{1-e^{-t}}\bar g_{p\bar d})=-\frac{1}{e^t-1}\tilde g^{\bar dk}(\tilde\nabla_i-\nabla^E_i)\bar g_{p\bar d}-\frac{1}{e^t-1}\tilde g^{\bar dk}\nabla^E_i\bar g_{p\bar d}\nonumber\\
&=\frac{1}{e^t-1}\tilde g^{\bar dk}T_{ip}^{e}\bar g_{e\bar d}-\frac{1}{e^t-1}\tilde g^{\bar dk}\nabla^E_i\bar g_{p\bar d}\nonumber\\
\end{align}
Note that for some constant $C\ge1$ we have $|\nabla^E\bar\omega|_{\omega_E}\le C$ on $K''\times[1,\infty)$.
Then by plugging \eqref{c8} into \eqref{c7} we see that
\begin{align}\label{c9}
\Delta_{\tilde\omega}S\ge|\tilde\nabla T|_{\tilde \omega}^2+|\overline{\tilde{\nabla}} T|_{\tilde\omega}^2-CS-C,
\end{align}
and hence
\begin{align}\label{c9.1}
\Delta_{\tilde\omega}(\varrho^2S)\ge\varrho^2(|\tilde\nabla T|_{\tilde \omega}^2+|\overline{\tilde{\nabla}} T|_{\tilde\omega}^2-CS)-CS+2Re\langle\tilde\nabla\varrho^2,\overline{\tilde\nabla}S\rangle_{\tilde\omega}\ge-CS,
\end{align}
where we have used the following
$$2Re\langle\tilde\nabla\varrho^2,\overline{\tilde\nabla}S\rangle_{\tilde\omega}\ge-C\varrho|\langle\tilde\nabla\varrho,\overline{\tilde\nabla}|T|_{\tilde\omega}^2\rangle_{\tilde\omega}|\ge-C\varrho|\tilde\nabla|T|_{\tilde\omega}^2|_{\tilde\omega}\ge-\varrho^2(|\tilde\nabla T|_{\tilde\omega}^2+|\overline{\tilde\nabla}T|_{\tilde\omega}^2)-CS$$
Also recall that, using \eqref{c3} and the Ricci lower bound given by \eqref{c5}, we have
\begin{align}\label{c10}
\Delta_{\tilde\omega}tr_{\tilde\omega}\omega_E&\ge-C+\tilde g^{\bar ji}\tilde g^{\bar qp} g_E^{\bar ba}\tilde\nabla_ig_{E,p\bar b}\tilde\nabla_{\bar j}g_{E,a\bar q}\nonumber\\
&=-C+\tilde g^{\bar ji}\tilde g^{\bar qp} g_E^{\bar ba}(-T_{ip}^eg_{E,e\bar b})(-\bar T_{ ja}^fg_{E,f\bar q})\nonumber\\
&\ge CS-C.
\end{align}
Now we combine \eqref{c9.1} and \eqref{c10} to see that, for a sufficiently large constant $A$,
\begin{equation}\label{c11}
\Delta_{\tilde\omega}(\varrho^2S+Atr_{\tilde\omega}\omega_E)\ge S-C.
\end{equation}
For any $t\in[1,\infty)$, we choose a maximal point $x_t\in \bar K''$ of $\varrho^2S+Atr_{\tilde\omega}\omega_E$. If $x_t\in\partial K''$, then
$$(\varrho^2S+Atr_{\tilde\omega}\omega_E)(x_t)\le Atr_{\tilde\omega}\omega_E(x_t)$$
is uniformly bounded from above; if $x_t\in K''$, then by applying the maximum principle to \eqref{c11} we also have $(\varrho^2S+Atr_{\tilde\omega}\omega_E)(x_t)$ is uniformly bounded from above. In conclusion
$$\sup_{K''\times[1,\infty)}(\varrho^2S+Atr_{\tilde\omega}\omega_E)\le C$$
for some uniform constant $C\ge1$. Therefore, using the fact that $\varrho\equiv1$ on $K'$, we see
$$\sup_{K'\times[1,\infty)}S\le C.$$
Now, we restrict $S$ to fiber $X_y$ and use the fact that $\tilde\omega|_{X_y}$ is equivalent to $\omega_{0,y}$ to conclude that
\begin{equation}\label{c12}
|e^t\omega(t)|_{X_y}|_{C^1(X_y,\omega_{0,y})}\le C.
\end{equation}
Moreover, from the arguments we easily see that the above bound $C$ can be chosen to be uniform when $y$ varies in a compact subset of $Y\setminus S$.
\par Lemma \ref{lemconv3} is proved.
\end{proof}

Before next step, we recall the following two facts from \cite[Section 2.3]{TWY}.

\begin{lem}\label{1lem}\cite[Lemma 2.4]{TWY}
For any given $K\subset\subset Y\setminus S$, consider a smooth function $P: f^{-1}(K) \times [1,\infty) \rightarrow \mathbb{R}$ which satisfies the following conditions:
\begin{itemize}
\item[(a)] There is a constant $A$ such that for all $y\in K$ and all $t\geq 1$
$$|\nabla (P|_{X_y})|_{\omega_{0,y}}\leq A.$$
\item[(b)] For all $y\in K$ and all $t\geq 1$ we have
$$\int_{x \in X_y}P(x,t) \bar\omega_{0,b}^n=0.$$
\item[(c)] There exists a function $G: [1, \infty) \rightarrow [0,\infty)$ such that $G(t) \rightarrow 0$ as $t \rightarrow \infty$ such that
$$\sup_{y\in f^{-1}(K)} P(y,t)\leq G(t),$$
for all $t\geq 1$.
\end{itemize}
Then  there is a constant $C$ such that $$\sup_{y\in f^{-1}(K)} |P(y,t)|\leq C G(t)^{\frac{1}{2n+1}}$$ for all $t$ sufficiently large.
\end{lem}

\begin{lem}\label{2lem}\cite[Lemma 2.6]{TWY}
Let $A$ be an $N \times N$ positive definite Hermitian symmetric matrix.  Assume that there exists $\epsilon\in (0,1)$ with
$$  \emph{tr} \, A   \le N + \epsilon, \quad   \det A   \ge 1-\epsilon.$$
Then there exists a constant $C_N$ depending only on $N$ such that
$$\| A- I \| \le C_N \sqrt{\epsilon},$$
where $\| \cdot \|$ is the Hilbert-Schmidt norm, and $I$ is the $N\times N$ identity matrix.
\end{lem}

Now we first prove the convergence of metric on smooth fibers. We point out that in \cite{TWY} one is given a Calabi-Yau fibration and hence there exists a canonical metric, i.e. Ricci-flat metric, in the initial class on any smooth Calabi-Yau fiber. Moreover, a key step in \cite{TWY} is to prove the restriction to a smooth Calabi-Yau fiber of a suitably normalized equation (namely, K\"ahler-Ricci flow or degeneration of Ricci-flat metrics) will converge to the Ricci-flat metric. In our case, the smooth fiber $X_y$ is a Fano manifold and we may not have a canonical metric on $X_y$. However, we do have defined an $\overline\omega_0$ in the initial class $[\omega_0]$, which is naturally associated to $\omega_0$ in the sense that $Ric(\overline\omega_{0,y})=\omega_{0,y}$. We now prove in the following Lemma \ref{local3} that, after restricting to a smooth fiber $X_y$, a suitably normalized continuity method will converge to $\overline\omega_{0,y}$. This is a key observation for later discussions.
\begin{lem}\label{local3}
For any given $K\subset\subset Y\setminus S$, there exists a positive function $G_1(t)$ such that for all $y\in K$ we have
\begin{equation}\label{local1'}
|e^t\omega(t)|_{X_y}-\overline\omega_{0,y}|_{C^0(X_y,\omega_{0,y})}\le G_1(t).
\end{equation}
\end{lem}

\begin{proof}
Write
\begin{align}
(e^t\omega(t)|_{X_y})^{n-k}&=\frac{(e^t\omega(t)|_{X_y})^{n-k}}{\overline\omega_{0,y}^{n-k}}\omega_{0,y}^{n-k}=e^{(n-k)t}\frac{\omega(t)^{n-k}\wedge f^*\omega_{Y}^k}{\overline\omega_{0}^{n-k}\wedge f^*\omega_{Y}^k}\overline\omega_{0,y}^{n-k}=\binom{n}{k}\frac{e^{(n-k)t}\omega(t)^{n-k}\wedge f^*\omega_{Y}^k}{e^{\psi}\Omega}\overline\omega_{0,y}^{n-k}\nonumber\\
&=\binom{n}{k}e^{\frac{\varphi(t)}{1-e^{-t}}-f^*\psi}\frac{\omega(t)^{n-k}\wedge f^*\omega_{Y}}{\omega(t)^n}\overline\omega_{0,y}^{n-k}\nonumber.
\end{align}
Then we define a function $P$ on $X\times[1,\infty)$ as follows:
$$P=\binom{n}{k}e^{\frac{\varphi(t)}{1-e^{-t}}-f^*\psi}\frac{\omega(t)^{n-k}\wedge f^*\omega_{Y}}{\omega(t)^n},$$
Easily,
$$P|_{X_y}=\frac{(e^t\omega(t))^{n-k}}{\overline\omega_{0,y}^{n-k}}.$$
We see that $\tilde P:=P-1$ satisfies the following properties:
\begin{itemize}
\item[(1)]$\int_{X_y}\tilde P\overline\omega_{0,y}^{n-k}=0$ for all $y\in K$;
\item[(2)]There exist a constant $C\ge1$ such that for all $y\in K$, $|\nabla\tilde P|_{X_y}|_{\omega_{0,y}}\le C$.
\item[(3)]As $t\to\infty$, $\sup_{f^{-1}(K)}\tilde P\le G_2(t)$ for some positive function $G_2(t)$ which converges to $0$ as $t\to\infty$.
\end{itemize}
Indeed, item (1) is obvious; item (2) follows from Lemma \ref{lemconv3} directly; for item (3), we use Lemma \ref{lemconv2} to see that
$$\binom{n}{k}\frac{\omega(t)^{n-k}\wedge f^*\omega_{Y}}{\omega(t)^n}\le\left(\frac{tr_{\omega(t)}f^*\omega_{Y}}{n-k}\right)^{n-k}\le 1+G_2(t),$$
and then, combining Lemma \ref{lemconv0}, we have item (3).
\par Using above items (1)-(3), we can conclude by Lemma \ref{1lem} that
$$\sup_{f^{-1}(K)}|\tilde P|\le G_4(t)$$
and hence
\begin{equation}\label{f1}
\|(e^t\omega(t)|_{X_y})^{n-k}-\overline\omega_{0,y}^{n-k}\|_{C^0(X_y,\omega_{0,y})}\le G_5(t).
\end{equation}
Next we define the following smooth function on $X_y$ for $y\in K$:
$$Q|_{X_y}=\frac{e^t\omega(t)|_{X_y}\wedge\overline\omega_{0,y}^{n-k-1}}{\overline\omega_{0,y}^{n-k}}.$$
Easily, $Q$ in fact smoothly depends in $y\in K$ and hence is a smooth function on $f^{-1}(K)$, which equals
$$Q=\frac{e^t\omega(t)\wedge\overline\omega_{0}^{n-k-1}\wedge f^*\omega_{Y}^k}{\overline\omega_{0}^{n-k}\wedge f^*\omega_{Y}^k}.$$
Set $\tilde Q=1-Q$, then $\tilde Q$ satisfies the following properties:
\begin{itemize}
\item[(1')]$\int_{X_y}\tilde Q\overline\omega_{0,y}^{n-k}=0$ for all $y\in K$;
\item[(2')]There exist a constant $C\ge1$ such that for all $y\in K$, $|\nabla\tilde Q|_{X_y}|_{\omega_{0,y}}\le C$.
\item[(3')]As $t\to\infty$, $\sup_{f^{-1}(K)}\tilde Q\le G_6(t)$ for some positive function $G_6(t)$ which converges to $0$ as $t\to\infty$.
\end{itemize}
Again, items (1') and (2') are clear to hold; for item (3'), we use the arithmetic-geometric means inequality to see that
$$(Q|_{X_y})^{n-k}\ge\frac{(e^t\omega(t)|_{X_y})^{n-k}}{\overline\omega_{0,y}^{n-k}}\ge1-G_7(t),$$
where in the last inequality we have used \eqref{f1}, so item (3') follows.
\par Therefore, using again Lemma \ref{1lem}, we have
$$\sup_{f^{-1}(K)}|\tilde Q|\le G_8(t),$$
which in particular implies that for all $y\in K$ we have
\begin{equation}\label{f2}
\|(e^t\omega(t)|_{X_y})\wedge\overline\omega_{0,y}^{n-k-1}-\overline\omega_{0,y}^{n-k}\|_{C^0(X_y,\omega_{0,y})}\le G_9(t).
\end{equation}
Now applying Lemma \ref{2lem} gives the desired conclusion \eqref{local1'}.
\par Lemma \ref{local3} is proved.
\end{proof}

We are ready to prove our main result in this section.
\begin{thm}\label{local4}
For any given $K\subset\subset Y\setminus S$, there exists a positive function $G_{10}(t)$ and a constant $T\ge1$ such that for all $t\in[T,\infty)$ we have
\begin{equation}\label{local4.1}
\|\omega(t)-f^*\omega_{Y}\|_{C^0(f^{-1}(K),\omega_0)}\le G_{10}(t).
\end{equation}
\end{thm}
\begin{proof}
We will make use of some arguments in \cite{TWY}. Set $\hat\omega(t):=e^{-t}\overline\omega_{0}+(1-e^{-t})f^*\omega_{Y}$ and choose a sufficiently large $T$ such that $\hat\omega(t)$ is a K\"ahler metric on $f^{-1}(K)$ for all $t\ge T$. Obviously, as $t\to\infty$,
\begin{equation}\label{f3.0}
\|\hat\omega(t)-f^*\omega_{Y}\|_{C^0(f^{-1}(K),\omega)}\to0.
\end{equation}
Moreover, after possibly increasing $T$, by Lemma \ref{lemconv0} we also have, for all $t\ge T$,
\begin{align}\label{f3}
\frac{\hat\omega(t)^n}{\omega(t)^n}&\ge\frac{\binom{n}{k}e^{-(n-k)t}\overline\omega_{0}\wedge f^*\omega_{Y}^k+o(e^{-(n-k+1)t})\Omega}{\omega(t)^n}\nonumber\\
&\ge e^{\frac{\varphi(t)}{1-e^{-t}}-f^*\psi}-Ce^{-t}\nonumber\\
&\ge1-G_{11}(t).
\end{align}
On the other hand, for any fix $x\in f^{-1}(K)$ and $y=f(x)\in K$ we write
\begin{align}\label{f3.1}
tr_{\omega(t)}\hat\omega(t)&=tr_{\omega(t)}f^*\omega_{Y}+tr_{\omega(t)}(e^{-t}\overline\omega_{0,y})+e^{-t}(-tr_{\omega(t)}f^*\omega_{Y}+tr_{\omega(t)}(\omega_{SRF}-\omega_{SRF,b})\nonumber\\
&\le tr_{\omega(t)}f^*\omega_{Y}+tr_{\omega(t)}(e^{-t}\overline\omega_{0,y})+e^{-t}tr_{\omega(t)}(\overline\omega_{0}-\overline\omega_{0,y})\nonumber\\
\end{align}
To bound the last term in \eqref{f3.1}, we now show that
\begin{equation}
|tr_{\omega(t)}(\overline\omega_{0}-\overline\omega_{0,y})|\le Ce^{\frac{t}{2}}.
\end{equation}
Indeed, if we choose a local chart $(U,z^1,\ldots,z^n)$ on $X$ centered at $x$ and a local chart $(V,z^{n-k+1},\ldots,z^{n})$ on $Y$ centered at $y$ such that $f$ is given by $(z^1,\ldots,z^n)\mapsto(z^{n-k+1},\ldots,z^n)$, then, since $\overline\omega_{0}-\overline\omega_{0,y}$ vanishes on fiber $X_y$, we can write
$$\overline\omega_{0}-\overline\omega_{0,y}=2Re\sqrt{-1}\sum_{\alpha=n-k+1}^n\sum_{i=1}^n\Psi_{\alpha\bar i}dz^{\alpha}\wedge d\bar z^{i}$$
for some smooth complex-valued functions $\Psi_{\alpha\bar i}$. Also note that by Lemma \ref{C2lem.4} and Cauchy-Schwarz inequality we have
$$|g^{\bar i\alpha}|\le Ce^{\frac{t}{2}}$$
whenever $n-k+1\le\alpha\le n$. Then
\begin{equation}\label{f4}
|tr_{\omega(t)}(\overline\omega_{0}-\overline\omega_{0,y})|=2|Re\sum_{\alpha=n-k+1}^n\sum_{i=1}^ng^{\bar i\alpha}\Psi_{\alpha\bar i}|\le Ce^{\frac{t}{2}}.
\end{equation}
Of course, the constant $C$ in the above inequality can be chosen to be uniform for all $x\in f^{-1}(K)$. Therefore, by plugging \eqref{f4} into \eqref{f3.1} and then using Lemmas \ref{lemconv2} and \ref{local3} we have
\begin{align}\label{f5}
tr_{\omega(t)}\hat\omega(t)&\le tr_{\omega(t)}f^*\omega_{Y}+tr_{\omega(t)}(e^{-t}\overline\omega_{0,y})+Ce^{-\frac{t}{2}}\nonumber\\
&=tr_{\omega(t)}f^*\omega_{Y}+tr_{(e^t\omega(t))}\overline\omega_{0,y}+Ce^{-\frac{t}{2}}\nonumber\\
&\le k+(n-k)+G_{12}(t)\nonumber\\
&=n+G_{12}(t)
\end{align}
Having \eqref{f3} and \eqref{f5}, we can apply Lemma \ref{2lem} to see that
$$\|\omega(t)-\hat\omega(t)\|_{C^0(f^{-1}(K),\omega(t))}\le G_{12}(t),$$
which, combining the fact by Lemma \ref{C2lem.4} that $\omega(t)\le C\omega_0$ on $f^{-1}(K)$ for some constant $C$, implies
\begin{equation}\label{f6}
\|\omega(t)-\hat\omega(t)\|_{C^0(f^{-1}(K),\omega_0)}\le G_{13}(t).
\end{equation}
Combining \eqref{f3.0} and \eqref{f6}, we have proved \eqref{local4.1}.
\par Theorem \ref{local4} is proved.
\end{proof}

\begin{proof}[Proof of Theorem \ref{t1}]
Combining Theorems \ref{L1conv} and \ref{local4}, Theorem \ref{t1} follows.
\end{proof}

\section*{Acknowledgements}
Y. Zhang is grateful to Professor Huai-Dong Cao for constant encouragement and support and Professor Chengjie Yu for constant help and invitation to visit Shantou University. Part of this work was carried out while Y. Zhang was visiting Capital Normal University and Shantou University, which he would like to thank for the warm hospitality. Both authors thank the referee for useful comments and suggestions.


\begin{thebibliography}{99}
\bibitem{Au} Aubin, T., \'{E}quations du type Monge-Amp\`{e}re sur les vari\'{e}t\'{e}s k\"ahleriennes compactes, C. R. Acad. Sci. Paris S\'{e}r. A-B 283 (1976), no. 3, Aiii, A119-A121
\bibitem{Au84} Aubin, T., R\'{e}duction de cas positif de l¡¯\'equation de Monge-Amp\`{e}re sur les vari\'{e}t\'{e}s k\"ahleriennes compactes \`{a} la d\'{e}monstration d¡¯une in\'{e}galit\'{e}. Journal of Functional Analysis (2) 57 (1984), 143-153
\bibitem{BM} Bando, S. and Mabuchi, T., Uniqueness of K\"ahler-Einstein metrics modulo connected group actions, in Algebraic Geometry, Sendai, 1985 (T. Oda, Ed.), Advanced Studies in Pure Mathematics 10, Kinokuniya, 1987, 11-40.
\bibitem{BHPV} Barth, W., Hulek, K., Peters, C. and Van de Ven, A., Compact complex surfaces, Springer (2003)
\bibitem{Cal} Calabi, E., On K\"ahler manifolds with vanishing canonical class, Algebraic Geometry and Topology, A Symposium in honor of S. Lefschatz, Princeton, 1955, 78-89
\bibitem{Ch} Choi, Y.-J., Semi-positivity of fiberwise Ricci-flat metrics on Calabi-Yau fibrations, arXiv:1508.00323v3
\bibitem{CLS} Cox, D., Little, J. and Schenck, H., Toric varieties, Graduate Studies in Mathematics, AMS, 124 (2010)
\bibitem{Cr} Croke, B., Some isoperimetric inequalities and eigenvalue estimates, Ann. Sci. \'Ecole. Norm. Sup. 13 (1980), 419-435 
\bibitem{DP} Demailly, J.-P. and Pali, N., Degenerate complex Monge-Amp\`{e}re equations over compact K\"{a}hler manifolds, Internat. J. Math. 21 (2010), no. 3, 357-405
\bibitem{EGZ08} Eyssidieux, P., Guedj, V. and Zeriahi, A., A priori $L^{\infty}$ estimates for Degenerate complex Monge-Amp\`{e}re equations, Int. Math. Res. Not. IMRN 2008, Art. ID rnn070, 8 pp
\bibitem{EGZ} Eyssidieux, P., Guedj, V. and Zeriahi, A., Singular K\"ahler-Einstein metrics, J. Amer. Math. Sci. 22 (2009), no. 3, 607-639
\bibitem{F1} Fong, F., K\"ahler-Ricci flow on projective bundles over K\"ahler-Einstein manifolds, Trans. Amer. Math. Soc. (2014) 366, 563-598
\bibitem{F2} Fong, F., On the collapsing rate of the K\"{a}hler-Ricci flow with finite time singularity, J. Geom. Anal. 25 (2015), no. 2, 1098-1107
\bibitem{FGS} Fu, X., Guo, B. and Song, J., Geometric estimates for complex Monge-Amp\`{e}re equation, arXiv: 1706.01527
\bibitem{FuZs} Fu, X. and Zhang, S., The K\"ahler-Ricci flow on Fano bundles, Math. Z. 286 (2017), no. 3-4, 1605-1626
\bibitem{K98} Ko{\l}odziej, S., The complex Monge-Amp\`{e}re equation, Acta Math. 180 (1), 69-117 (1998)
\bibitem{K05} Ko{\l}odziej, S., The complex Monge-Amp\`{e}re equation and pluripotential theory, Mem. Amer. math. Soc. 178 (2005), no. 840, x+64 pp
\bibitem{LT} La Nave, G. and Tian, G., A continuity method to construct canonical metrics, Math. Ann. 365 (2016), 911-921
\bibitem{LTZ} La Nave, G., Tian, G. and Zhang, Z.L., Bounding diameter of singular K\"ahler metric, Amer. J. Math. 139 (2017), no. 6, 1693-1731
\bibitem{La} Lazarafeld, J., Positivity in algebraic geometry. I, A Series of Modern Survays in Mathematics, 48. Springer-Verlag, Berlin, 2004
\bibitem{Ma} Matsuki, K., Introduction to the Mori program, Universitext, Springer-Verlag, New York, 2002
\bibitem{Ru} Rubinstein, Y., Some discretizations of geometric evolution equations and the Ricci iteration on the space of Kahler metrics, Adv. Math. 218 (2008), 1526-1565
\bibitem{SeT} \v{S}e\v{s}um, N. and Tian, G., Bounding scalar curvature and diameter along the K\"{a}hler-Ricci flow (after Perelman), J. Inst. Math. Jussieu 7 (2008), 575-587
\bibitem{ShWe} Sherman, M. and Weinkove, B., Interior derivative estimates for the K\"{a}hler-Ricci flow, Pacific J. Math. 257 (2012), 491-501
\bibitem{Si} Siu, Y.-T., Lectures on Hermitian-Einstein metrics for stable bundles and K\"ahler-Einstein metrics, Birkh\"auser, Basel (1987)
\bibitem{So} Song, J., Finite time extinction of the K\"{a}hler-Ricci flow, Math. Res. Lett. 21 (2014), no. 6, 1435-1449
\bibitem{SSW} Song, J., Sz\'{e}kelyhidi, G. and Weinkove, B.,  The K\"{a}hler-Ricci flow on projective bundles, Int. Math. Res. Not. 2013, no. 2, 243-257
\bibitem{ST06} Song, J. and Tian, G., The K\"{a}hler-Ricci flow on surfaces of positive Kodaira dimension, Invent. Math., 170, 609-653 (2006)
\bibitem{ST12} Song, J. and Tian, G., Canonical measures and K\"{a}hler-Ricci flow, J. Amer. Math. Soc. 25 (2012), no. 2, 303-353
\bibitem{ST16} Song, J. and Tian, G., The K\"{a}hler-Ricci flow through singularities, Invent. Math. 207 (2017), no. 2, 519-595
\bibitem{SW} Song, J. and Weinkove, B.,  The K\"{a}hler-Ricci flow on Hirzebruch surfaces, J. Reine Angew. Math., 659 (2011), 141-168
\bibitem{SY} Song, J. and Yuan, Y., Metric flips with Calabi ansatz, Geom. Funct. Anal. 22 (2012), 240-265
\bibitem{TZ} Tian, G. and Zhu, X., Convergence of K\"{a}hler-Ricci flow, J. Amer. Math. Soc. 20, 675-699 (2007)
\bibitem{Top} Topping, P., Relating diameter and mean curvature for submanifolds of Euclidean space, Comment. Math. Helv. 83 (2008), no. 3, 539-546
\bibitem{To10} Tosatti, V., Adiabatic limits of Ricci-flat K\"{a}hler metrics, J. Diff. Geom. 84 (2010) 427-453
\bibitem{TWY} Tosatti, V., Weinkove, B. and Yang, X., The K\"{a}hler-Ricci flow, Ricci-flat metrics and collapsing limits, Amer. J. Math. 140 (2018), no. 3, 653-698
\bibitem{ToZy15} Tosatti, V. and Zhang, Y.G., Infinite time singularities of  the K\"{a}hler-Ricci flow, Geom. Topol, 19 (2015), 2925-2948
\bibitem{ToZy16} Tosatti, V. and Zhang, Y.G., Finite time collapsing of the K\"{a}hler-Ricci flow on threefolds, Ann. Sc. Norm. Super. Pisa Cl. Sci. 18 (2018), no.1, 105-118
\bibitem{Y78} Yau, S.-T., A general Schwarz lemma for K\"ahler manifolds, Amer. J. Math. 100 (1978), no. 1, 197-203
\bibitem{Y} Yau, S.-T., On the Ricci curvature of a compact K\"{a}hler manifold and the complex Monge-Amp\`{e}re equation, I, Comm. Pure Appl. Math. 31 (1978) 339-411
\bibitem{ZyZz} Zhang, Y.S. and Zhang, Z.L., The continuity method on minimal elliptic K\"ahler surfaces, to appear in Int. Math. Res. Not., https://doi.org/10.1093/imrn/rnx209
\end{thebibliography}
\end{document}